\documentclass[graybox]{svmult}

\usepackage{mathptmx}       % selects Times Roman as basic font
\usepackage{helvet}         % selects Helvetica as sans-serif font
\usepackage{courier}        % selects Courier as typewriter font
\usepackage{type1cm}        % activate if the above 3 fonts are
\usepackage{makeidx}         % allows index generation
\usepackage{graphicx}        % standard LaTeX graphics tool
\usepackage{multicol}        % used for the two-column index
\usepackage[bottom]{footmisc}% places footnotes at page bottom
\smartqed

\makeindex             % used for the subject index
\usepackage{amsmath}
\usepackage{amssymb}
\usepackage[english]{babel}
\usepackage{setspace}
\allowdisplaybreaks[3]

\author{J. Arnlind \and A. Kitouni \and A. Makhlouf \and S. Silvestrov}
\title*{Structure and Cohomology of 3-Lie Algebras Induced by Lie Algebras}
\institute{J. Arnlind \at Department of Mathematics, Link\"{o}ping University, 581 83 Link\"{o}ping, Sweden,\\ \email{joakim.arnlind@liu.se}.
\and A. Kitouni \at Universit\'{e} de Haute-Alsace, 4 Rue des Fr\`{e}res Lumi\`{e}re, 68093 Mulhouse, France,\\ \email{abdennour.kitouni@uha.fr}
\and A. Makhlouf \at Universit\'{e} de Haute-Alsace, 4 Rue des Fr\`{e}res Lumi\`{e}re, 68093 Mulhouse, France,\\ \email{abdenacer.makhlouf@uha.fr}
\and S. Silvestrov \at M\"{a}lardalens h\"{o}gskola, Box 883, 721 23 V\"{a}ster\r{a}s, Sweden, \\ \email{sergei.silvestrov@mdh.se}
}

\newcommand{\AKMSbracket}[1]{\left[#1\right]}
\newcommand{\AKMSpara}[1]{\left(#1\right)}

\begin{document}

\maketitle

\abstract{
The aim of this paper is to compare the structure and the cohomology spaces of Lie algebras and induced $3$-Lie algebras. 
}

\section{Introduction}
Lie algebras have held a very important place in mathematics and physics for a long time. Ternary Lie algebras   appeared first  in Nambu's generalization of Hamiltonian mechanics \cite{Nambu:GenHD} which uses a generalization of Poisson algebra with a ternary bracket. The algebraic formulation is due to Takhtajan. The structure of $n$-Lie algebra was studied by Filippov \cite{Filippov:nLie} and Kasymov \cite{Kasymov:nLie}. %%%

The Lie algebra cohomology complex is well known under the name of Chevalley-Eilenberg cohomology complex. The cohomology of $n$-Lie algebras was first introduced by Takhtajan \cite{Takhtajan:cohomology} in its simplest form, later a complex adapted to the study of formal deformations was introduced by Gautheron \cite{Gautheron:Rem}, then reformulated by Daletskii and Takhtajan \cite{Dal_Takh} using the notion of base Leibniz algebra of a $n$-Lie algebra.

In \cite{almy:quantnambu}, the authors introduce a realization of the quantum Nambu bracket in terms of matrices (using the commutator and the trace of matrices). This construction is generalized in \cite{ams:ternary} to the case of any Lie algebra where the commutator is replaced by the Lie bracket, and the matrix trace is replaced by linear forms having similar properties, which we call $3$-Lie algebras induced by Lie algebras. This construction is generalized to the case of $n$-Lie algebras in \cite{ams:n}. 
We investigate the connections between the structural properties (Solvability, nilpotency,...) and the cohomology of a Lie algebra and an induced $3$-Lie algebra.

The paper is organized as follows: in Section \ref{AKMS:n-Lie} we recall main definitions and results concerning $n$-Lie algebras, and construction of $(n+1)$-Lie algebras induced by $n$-Lie algebras. In Section \ref{AKMS:Structure} we study some structural properties of $3$-Lie algebras induced by Lie algebras, in particular: common subalgebras and ideals, solvability and nilpotency. In Section \ref{AKMS:Cohomology}, we recall the cohomology complexes for Lie algebras and $3$-Lie algebras, then we study relations between $1$ and $2$ cocycles of a Lie algebra and the induced $3$-Lie algebra. In Section \ref{AKMS:Extension}, we give definitions of central extensions of Lie algebras and $n$-Lie algebras, then we study the relation between central extension of a Lie algebra and those of a $3$-Lie algebra it induces. In Section \ref{AKMS:Classifications} we give a method to recognize $3$-Lie algebras that are induced by some Lie algebra, and applying it, we can determine all $3$-Lie algebras induced by Lie algebras up to dimension $5$, based on classifications given in \cite{Bai:nLie:n+2} and \cite{Filippov:nLie}, then we give a list of Lie algebras up to dimension $4$ and all the possible induced $3$-Lie algebras. In Section \ref{AKMS:Examples} we present compute on 4 chosen Lie algebras and one trace map each, the set of $1$-cocycles and $1$ coboundaries of the Lie algebras and the induced $3$-Lie algebras using the computer algebra software Mathematica, the algorithm is briefly explained there too.

\section{$n$-Lie Algebras}\label{AKMS:n-Lie}
In this paper, all considered vector spaces are over a field $\mathbb{K}$ of characteristic $0$.
$n$-Lie algebras were introduced in \cite{Filippov:nLie}, then deeper investigated in \cite{Kasymov:nLie}. Let us recall of some basic definitions.
  \begin{definition}
    A  $n$-Lie algebra \index{$n$-Lie algebra} $\AKMSpara{A,\AKMSbracket{\cdot,...,\cdot}}$ is a vector space together with a skew-symmetric $n$-linear map $\AKMSbracket{\cdot,...,\cdot} : A^n \to A$ such that :
    \begin{equation}
 \AKMSbracket{x_1,...,x_{n-1},\AKMSbracket{y_1,...,y_n}} = \sum_{i=1}^n \AKMSbracket{y_1,...,\AKMSbracket{x_1,...,x_{n-1},y_i},...,y_n}. \label{AKMS-FI}
    \end{equation}
   for all $x_1,...,x_{n-1},y_1,...,y_n \in A$. This condition is called the fundamental identity or Filippov identity. For $n=2$ (\ref{AKMS-FI}) becomes the Jacobi identity and we get the definition of a Lie algebra.
  \end{definition} 

\begin{definition}
Let $(A,\AKMSbracket{\cdot,...,\cdot})$ be a $n$-Lie algebra, and $I$ a subspace of $A$. We say that $I$ is an ideal of $A$ if, for all $i\in I, x_1,...,x_{n-1}\in A$, it holds that $\AKMSbracket{i,x_1,...,x_{n-1}}\in I$.
\end{definition}

\begin{lemma}
Let $(A,\AKMSbracket{\cdot,...,\cdot})$ be a $n$-Lie algebra, and $I_1,....,I_n$ be ideals of $A$, then $I=\AKMSbracket{I_1,...,I_n}$ is an ideal of $A$.
\end{lemma}

\begin{definition}
Let $(A,\AKMSbracket{\cdot,...,\cdot})$ be a $n$-Lie algebra, and $I$ an ideal of $A$. Define the derived series of $I$ by:
\[D^0(I)=I \text{ and } D^{p+1}(I)=\AKMSbracket{D^p(I),...,D^p(I)}.\]
and the central descending series of $I$ by:
\[C^0(I)=I \text{ and } C^{p+1}(I)=\AKMSbracket{C^p(I),I,...,I}.\]
\end{definition}

\begin{definition}
Let $(A,\AKMSbracket{\cdot,...,\cdot})$ be a $n$-Lie algebra, and $I$ an ideal of $A$. $I$ is said to be solvable if there exists $p \in \mathbb{N}$ such that $D^p(I)=\{0\}$. It is said to be nilpotent if there exists $p \in \mathbb{N}$ such that $C^p(I)=\{0\}$.
\end{definition}

\begin{definition}
A $n$-Lie algebra $(A,\AKMSbracket{\cdot,...,\cdot})$ is said to be simple if $D^1(A) \neq \{ 0 \}$ and if it has no ideals other than $\{0\}$ and $A$. A direct sum of simple $n$-Lie algebras is said to be semi-simple.
\end{definition}

In \cite{ams:ternary} and \cite{ams:n} a construction of a 3-Lie algebra from a Lie algebra, and more generally a $(n+1)$-Lie algebra from a $n$-Lie algebra was introduced. We recall the main definitions and results.

\begin{definition}
  Let $\phi:A^n\to A$ be a $n$-linear map and let $\tau$ be a linear map from
  $A$ to $\mathbb{K}$. Define $\phi_\tau:A^{n+1}\to A$ by:
  \begin{align}
    \phi_\tau(x_1,...,x_{n+1}) = \sum_{k=1}^{n+1}(-1)^k\tau(x_k)\phi(x_1,...,\hat{x}_k,...,x_{n+1}),
  \end{align}
where the hat over $\hat{x}_k$ on the right hand side means that $x_{k}$ is excluded, that is 
 $\phi$ is calculated on $(x_1,\ldots, x_{k-1}, x_{k+1}, ... , x_{n+1})$. 
\end{definition}
We will not be concerned with just any linear map $\tau$,
but rather maps that have a generalized trace property. Namely:

\begin{definition}
  For $\phi:A^n\to A$ we call a linear map $\tau:A \to \mathbb{K}$ a
  \emph{$\phi$-trace (or trace)} if $\tau\AKMSpara{\phi(x_1,\ldots,x_n)}=0$ for all
  $x_1,\ldots,x_n\in A$.
\end{definition}

\begin{lemma}
  Let $\phi:A^n\to A$ be a skew-symmetric $n$-linear map and
  $\tau$ a linear map $A\to\mathbb{K}$. Then $\phi_\tau$ is a $(n+1)$-linear
  totally skew-symmetric map. Furthermore, if $\tau$ is a $\phi$-trace
  then $\tau$ is a $\phi_\tau$-trace.
\end{lemma}

\begin{theorem}
Let $(A,\phi)$ be a $n$-Lie algebra and $\tau$ a $\phi$-trace, then $(A,\phi_\tau)$ is a $(n+1)$-Lie algebra. We shall say that $(A,\phi_\tau)$ is induced by $(A,\phi)$. In particular, let $(A,\AKMSbracket{.,.})$ be a Lie algebra and $\tau : A \to \mathbb{K}$ be a trace map, the ternary bracket $\AKMSbracket{.,.,.}$ given by: $\AKMSbracket{x,y,z}=\underset{x,y,z}{\LARGE{\circlearrowleft }}\tau \AKMSpara{x} \AKMSbracket{y,z}$ defines a $3$-Lie algebra, we refer to $A$ when considering the Lie algebra and $A_\tau$ when considering induced $3$-Lie algebra.
\end{theorem}

%%%%%%%%%%%%%%%%%%%%%%%%%%%%%%%%%%%%%%%%%%%%%%%%%%%%%%

\section{Structure of 3-Lie Algebras Induced by Lie Algebras}\label{AKMS:Structure}
Let $(A,\AKMSbracket{.,.})$ be a Lie algebra, $\tau$ a $\AKMSbracket{.,.}$-trace and $(A,\AKMSbracket{.,.,.}_\tau)$ the induced 3-Lie algebra.
\begin{proposition}
If $B$ is a subalgebra of $(A,\AKMSbracket{.,.})$ then $B$ is also a subalgebra of $(A,\AKMSbracket{.,.,.}_\tau)$.
\end{proposition}
\begin{proof}
Let $B$ be a subalgebra of $(A,\AKMSbracket{.,.})$ and $x,y,z \in B$:
\[ \AKMSbracket{x,y,z}_\tau = \tau(x)\AKMSbracket{y,z}+\tau(y)\AKMSbracket{z,x}+\tau(z)\AKMSbracket{x,y},\]
which is a linear combination of elements of $B$ and then belongs to $B$.\qed
\end{proof}

\begin{proposition}
Let $J$ be an ideal of  $(A,\AKMSbracket{.,.})$. Then $J$ is an ideal of $(A,\AKMSbracket{.,.,.}_\tau)$ if and only if :
\[  \AKMSbracket{A,A} \subseteq J \text{ or } J \subseteq \ker \tau. \]
\end{proposition}
\begin{proof}
Let $J$ be an ideal of $(A,\AKMSbracket{.,.})$, and let $j \in J$ and $x,y \in A$, then we have:
\[ \AKMSbracket{x,y,j}_\tau = \tau(x) \AKMSbracket{y,j} + \tau(y) \AKMSbracket{j,x} + \tau(j) \AKMSbracket{x,y}. \]
We have that $\tau(x) \AKMSbracket{y,j} + \tau(y) \AKMSbracket{j,x} \in J$, then, to have $\AKMSbracket{x,y,j}_\tau \in J$ it is necessary and sufficient to have $\tau(j) \AKMSbracket{x,y} \in J$, which is equivalent to $\tau(j) = 0$ or $\AKMSbracket{x,y} \in J$. \qed
\end{proof}

\begin{theorem} \label{AKMSsolv2}%%%
Let $(A,\AKMSbracket{.,.})$ be a Lie algebra, $\tau$ a $\AKMSbracket{.,.}$-trace and $(A,\AKMSbracket{.,.,.}_\tau)$ the induced 3-Lie algebra. The $3$-Lie algebra $(A,\AKMSbracket{.,.,.}_\tau)$ is solvable, more precisely $D^2(A_\tau) = 0$ i.e. $\left(D^1(A_\tau)=\AKMSbracket{A,A,A}_\tau,\AKMSbracket{.,.,.}_\tau\right)$ is abelian.
\end{theorem}
\begin{proof}
Let $x,y,z\in \AKMSbracket{A,A,A}_\tau$, $x=\AKMSbracket{x_1,x_2,x_3}_\tau$,  $y=\AKMSbracket{y_1,y_2,y_3}_\tau$ and  $z=\AKMSbracket{z_1,z_2,z_3}_\tau$, then:
\begin{align*}
\AKMSbracket{x,y,z}_\tau &= \tau\AKMSpara{\AKMSbracket{x_1,x_2,x_3}_\tau}\AKMSbracket{\AKMSbracket{y_1,y_2,y_3}_\tau,\AKMSbracket{z_1,z_2,z_3}_\tau}\\
 &+  \tau\AKMSpara{\AKMSbracket{y_1,y_2,y_3}_\tau}\AKMSbracket{\AKMSbracket{z_1,z_2,z_3}_\tau,\AKMSbracket{x_1,x_2,x_3}_\tau}\\
 &+ \tau\AKMSpara{\AKMSbracket{z_1,z_2,z_3}_\tau}\AKMSbracket{\AKMSbracket{x_1,x_2,x_3}_\tau,\AKMSbracket{y_1,y_2,y_3}_\tau} \\
 &= 0. 
\end{align*}
Because $\tau\left([.,.,.]\right)=0$.
\qed
\end{proof}

\begin{remark}[\cite{Filippov:nLie}]
Let $(A,\AKMSbracket{.,.,.})$ be a $3$-Lie algebra. If we fix $a \in A$, the bracket \[ \AKMSbracket{.,.}_a = \AKMSbracket{a,.,.} \] is skew-symmetric and satisfies Jacobi identity. Indeed, we have, for $x,y,z \in A$:
\begin{align*}
\AKMSbracket{x,\AKMSbracket{y,z}_a}_a &= \AKMSbracket{a,x,\AKMSbracket{a,y,z}}\\
&= \AKMSbracket{\AKMSbracket{a,x,a},y,z} + \AKMSbracket{a,\AKMSbracket{a,x,y},z} + \AKMSbracket{a,y,\AKMSbracket{a,x,z}}\\
&= \AKMSbracket{a,\AKMSbracket{a,x,y},z} + \AKMSbracket{a,y,\AKMSbracket{a,x,z}}\\
&= \AKMSbracket{\AKMSbracket{x,y}_a,z}_a  + \AKMSbracket{x,\AKMSbracket{y,z}_a}_a
\end{align*}

\end{remark}

\begin{proposition}
Let $\AKMSpara{A,\AKMSbracket{.,.}}$ be a Lie algebra, $\tau$ be a trace and $\AKMSpara{A,\AKMSbracket{.,.,.}_\tau}$ the induced algebra, let $\AKMSpara{C^p(A)}$ be the central descending series of $\AKMSpara{A,\AKMSbracket{.,.}}$, and $\AKMSpara{C^p(A_\tau)}$ be the central descending series of $\AKMSpara{A,\AKMSbracket{.,.,.}_\tau}$. then we have : 
\[ C^p(A_\tau) \subset C^p(A), \forall p \in \mathbb{N}. \]
If there exists $i \in A$ such that $\AKMSbracket{i,x,y}_\tau = \AKMSbracket{x,y}, \forall x,y \in A$ then:
\[ C^p(A_\tau) = C^p(A), \forall p \in \mathbb{N}. \]
\end{proposition}
\begin{proof}
We proceed by induction over $p \in \mathbb{N}$. The case of $p=0$ is trivial, for $p=1$ we have:
\[\forall x = \AKMSbracket{a,b,c}_\tau \in C^1(A_\tau), x=\tau(a) \AKMSbracket{b,c} + \tau(b) \AKMSbracket{c,a} + \tau(c) \AKMSbracket{a,b}, \] which is a linear combination of elements of $C^1(A)$ and then is an elements of $C^1(A)$.
Suppose now that there exists $i \in A$ such that $\AKMSbracket{i,x,y}_\tau = \AKMSbracket{x,y}, \forall x,y \in A$, then for $x = \AKMSbracket{a,b} \in C^1(A)$, $x=\AKMSbracket{i,a,b}_\tau$ and then is an element of $C^1(A_\tau)$.

Now, we suppose this proposition is true for some $p \in \mathbb{N}$, and let $x \in C^{p+1}(A_\tau)$, then $x=\AKMSbracket{a,u,v}_\tau$ with $u,v \in A$ and $a \in C^{p}(A_\tau)$
\[x=\AKMSbracket{a,u,v}_\tau = \tau(u)\AKMSbracket{v,a} + \tau(v) \AKMSbracket{a,u} \qquad (\tau(a)=0)\]
which is an element of $C^{p+1}(A)$ because $a \in C^{p}(A_\tau) \subset C^p(A)$.
If there exists $i \in A$ such that $\AKMSbracket{i,x,y}_\tau = \AKMSbracket{x,y}, \forall x,y \in A$ then, if $x \in C^{p+1}(A)$ then $x = \AKMSbracket{a,u}$ with $a \in C^{p}(A)$ and $u \in A$ and we have: \[ x = \AKMSbracket{a,u} = \AKMSbracket{i,a,u}_\tau = \AKMSbracket{a,u,i}_\tau \in C^{p+1}(A_\tau).\] \qed 
\end{proof}

\begin{remark} \label{AKMS-D3subD}
It also results from the preceding proposition that $D^1(A_\tau) = \AKMSbracket{A,A,A}_\tau \subset D^1(A) = \AKMSbracket{A,A}$, and that if there exists $i \in A$ such that $\AKMSbracket{i,x,y}_\tau = \AKMSbracket{x,y}, \forall x,y \in A$, then $D^1 (A_\tau) = D^1(A)$. For the rest of the derived series, we have obviously the first inclusion by Theorem \ref{AKMSsolv2}.
\end{remark}

\begin{theorem}
Let $\AKMSpara{A,\AKMSbracket{.,.}}$ be a Lie algebra, $\tau$ be a trace and $\AKMSpara{A,\AKMSbracket{.,.,.}_\tau}$ the induced algebra, then we have :
\[ \AKMSpara{A,\AKMSbracket{.,.}} \text{is nilpotent of class } p \implies \AKMSpara{A,\AKMSbracket{.,.,.}_\tau} \text{is nilpotent of class at most } p. \]
Moreover, if there exists $i \in A$ such that $\AKMSbracket{i,x,y}_\tau = \AKMSbracket{x,y}, \forall x,y \in A$ then:
\[ \AKMSpara{A,\AKMSbracket{.,.}} \text{is nilpotent of class } p \iff \AKMSpara{A,\AKMSbracket{.,.,.}_\tau} \text{is nilpotent of class } p. \]
\end{theorem}
\begin{proof} 
\begin{enumerate}
\item Suppose that $\AKMSpara{A,\AKMSbracket{.,.}}$ is nilpotent of class $p\in \mathbb{N}$, then $C^p(A)=\{0\}$. By the preceding proposition, $C^p (A_\tau) \subseteq C^p(A)=\{0\}$, therefore $\AKMSpara{A,\AKMSbracket{.,.,.}_\tau}$ is nilpotent of class at most $p$.
\item We suppose now that $\AKMSpara{A,\AKMSbracket{.,.,.}_\tau}$ is nilpotent of class $p \in \mathbb{N}$, and that there exists $i \in A$ such that $\AKMSbracket{i,x,y}_\tau = \AKMSbracket{x,y}, \forall x,y \in A$, then $C^p(A_\tau)=\{0\}$. By the preceding proposition, $C^p(A) = C^p(A_\tau)=\{0\}$. Therefore $\AKMSpara{A,\AKMSbracket{.,.}}$ is nilpotent, since $C^{p-1}(A) = C^{p-1}(A_\tau) \neq \{0\}$,  $\AKMSpara{A,\AKMSbracket{.,.,.}_\tau}$ and $\AKMSpara{A,\AKMSbracket{.,.}}$ have the same nilpotency class.
\end{enumerate}\qed
\end{proof}

%%%%%%%%%%%%%%%%%%%%%%%%%%%%%%%%%%%%%%%%%%%%%%%%%%%%%%

\section{Lie and $3$-Lie Algebras Cohomology}\label{AKMS:Cohomology}
In this section, we study the connections between  the Chevalley-Eilenberg cohomology for Lie algebras and the cohomology of $3$-Lie algebras induced by Lie algebras.

Now, let us recall the main definitions of Lie algebras and $n$-Lie algebras cohomology, for reference and further details, see \cite{Dal_Takh}, \cite{aip:review}, \cite{Gautheron:Rem} and \cite{Takhtajan:cohomology}.

\begin{definition}
Let $(A,\AKMSbracket{.,.})$ be a Lie algebra, $\rho$ a representation of $A$ in a vector space $M$. A $M$-valued $p$-cochain on $A$ is a skew-symmetric $p$-linear map $\varphi : A^p \to M$, the set of $M$-valued $p$-cochain is denoted by $C^p(A,M)$.

The coboundary operator is the linear map $\delta^p : C^p(A,M) \to C^{p+1}(A,M)$ given by :
\begin{align*}
\delta^p\varphi(x_1,...,x_{p+1}) &= \sum_{j=1}^{p+1} (-1)^{j+1} \rho (x_k)\varphi(x_1,...,\hat{x}_j,...,x_{p+1}) \\
&+ \sum_{j=1}^{p+1} \sum_{k=j+1}^{p+1} (-1)^{j+k}\varphi([x_j,x_k],x_1,...,\hat{x}_j,...,\hat{x}_k,...,x_{p+1}).
\end{align*}

We will study two particular cases, the adjoint cohomology $M=A, \rho = ad$ and the scalar cohomology $M = \mathbb{K}, \rho = 0$.
\end{definition}
\begin{definition}
Let $(A,[.,.,.])$ be a $3$-Lie algebra, an $A$-valued $p$-cochain is a linear map $\psi : (\wedge^2 A)^{\otimes p-1} \wedge A \to A$.
\end{definition}
\begin{definition}
 The coboundary operator for the adjoint action is given by :
\begin{align*}
d^p\psi(x_1,...,x_{2p+1}) &= \sum_{j=1}^p \sum_{k=2j+1}^{2p+1} (-1)^j \psi (x_1,...,\hat{x}_{j-1},\hat{x}_j,...ad_{a_j}x_k,...,x_{2p+1}) \\
&+ \sum_{k=1}^{p} (-1)^{k-1} ad_{a_k}\psi(x_1,...,\hat{x}_{2k-1},\hat{x}_{2k},...,x_{2p+1}) \\
&+ (-1)^{p+1} \AKMSbracket{x_{2p-1},\psi(x_1,...,x_{2p-2},x_{2p}),x_{2p+1}}\\
&+ (-1)^{p+1} \AKMSbracket{\psi(x_1,...,x_{2p-1},x_{2p},x_{2p+1}},
\end{align*}
where $a_k=(x_{2k-1},x_{2k})$.
\end{definition}

\begin{definition}
Let $(A,[.,.,.])$ be a $3$-Lie algebra, a $\mathbb{K}$-valued $p$-cochain is a linear map $\psi : (\wedge^2 A)^{\otimes p-1} \wedge A \to \mathbb{K}$.
\end{definition}
\begin{definition}
 The coboundary operator for the trivial action is given by :
\begin{align*}
d^p\psi(x_1,...,x_{2p+1}) =& \sum_{j=1}^p \sum_{k=2j+1}^{2p+1} (-1)^j \psi (x_1,...,\hat{x}_{j-1},\hat{x}_j,...ad_{a_j}x_k,...,x_{2p+1}), \\
\end{align*}
where $a_k=(x_{2k-1},x_{2k})$.
\end{definition}

%\begin{definition}
The elements of $Z^p(A,M) = \ker \delta^p$ are called $p$-cocycles, those of $B^n(A,M)= \operatorname{Im} \delta^{p-1}$ are called coboundaries. $H^p(A,M)=\frac{Z^p(A,M)}{B^n(A,M)}$ is the $p$-th cohomology group. We sometimes add in subscript the representation used in the cohomology complex, for example $Z_{ad}^p(A,A)$ denotes the set of $p$-cocycle for the adjoint cohomology and $Z_{0}^p(A,\mathbb{K})$  denotes the set of $p$-cocycle for the scalar cohomology.

In particular, the elements of $Z^1(A,A)$ are the derivations. Recall that a derivation of a $n$-Lie algebra is a linear map $f : A \to A$ satisfying:
\[ f\AKMSpara{\AKMSbracket{x_1,...,x_n}} = \sum_{i=1}^n \AKMSbracket{x_1,...,f(x_i),...,x_n},\forall x_1,...,x_n \in A. \]
%\end{definition}

\subsection{Derivations and 2-cocycles correspondence}
Let $(A,\AKMSbracket{.,.})$ be a Lie algebra, $\tau$ a $\AKMSbracket{.,.}$-trace and $(A,\AKMSbracket{.,.,.}_\tau)$ the induced 3-Lie algebra, then we have the following correspondence between 1 and 2-cocycles of $(A,\AKMSbracket{.,.})$ and those of $(A,\AKMSbracket{.,.,.}_\tau)$.

\begin{lemma}
Let $f : A \to A$ be a Lie algebra derivation, then $\tau \circ f$ is a $\AKMSbracket{.,.}$-trace. 
\end{lemma}
\begin{proof}
for all $x,y \in A$, we have :
\begin{align*}
\tau\AKMSpara{ f\AKMSpara{ \AKMSbracket{x,y} } } &= \tau\AKMSpara{ \AKMSbracket{f(x),y} + \AKMSbracket{x,f(y)} }
=\tau\AKMSpara{ \AKMSbracket{f(x),y} } + \tau\AKMSpara{ \AKMSbracket{x,f(y)} } = 0.
\end{align*} \qed
\end{proof}

\begin{theorem}
Let $f:A \to A$ be a derivation of the Lie algebra $A$, then $f$ is a derivation of the induced 3-Lie algebra if and only if:
\[ \AKMSbracket{x,y,z}_{\tau \circ f} =0, \forall x,y,z \in A. \]

\end{theorem}

\begin{proof}
Let $f$ be a derivation of $A$ and $x,y,z \in A$:
\begin{align*}
f\AKMSpara{\AKMSbracket{x,y,z}_\tau} &= \tau(x)f\AKMSpara{\AKMSbracket{y,z}}+\tau(y)f\AKMSpara{\AKMSbracket{z,x}}+\tau(z)f\AKMSpara{\AKMSbracket{x,y}} \\
&= \tau(x)\AKMSbracket{f(y),z}+\tau(y)\AKMSbracket{f(z),x}+\tau(z)\AKMSbracket{f(x),y}\\
&+ \tau(x)\AKMSbracket{y,f(z)}+\tau(y)\AKMSbracket{z,f(x)}+\tau(z)\AKMSbracket{x,f(y)}\\
&+ \tau(f(x))\AKMSbracket{y,z}+\tau(f(y))\AKMSbracket{z,x}+\tau(f(z))\AKMSbracket{x,y}\\
&- \tau(f(x))\AKMSbracket{y,z}+\tau(f(y))\AKMSbracket{z,x}+\tau(f(z))\AKMSbracket{x,y}\\
&= \AKMSbracket{f(x),y,z}_\tau + \AKMSbracket{x,f(y),z}_\tau + \AKMSbracket{x,y,f(z)}_\tau - \AKMSbracket{x,y,z}_{\tau \circ f}. \qed
\end{align*}
\end{proof}

\begin{theorem}\label{AKMS-Z2ad}
Let $\varphi \in Z^2_{ad}(A,A)$ and $\omega : A \to \mathbb{K}$ be a linear map satisfying :
\begin{enumerate}
\item $\tau(x)\omega(y) = \tau(y)\omega(x)$,
\item $\omega([x,y])=0$,
\item $\underset{x,y,z}{\LARGE{\circlearrowleft }}\omega \AKMSpara{x} \tau \AKMSpara{ \varphi \AKMSpara{ y,z} } = 0$.
\end{enumerate}
Then $\psi\AKMSpara{x,y,z} = \underset{x,y,z}{\LARGE{\circlearrowleft }}\omega \AKMSpara{x} \varphi \AKMSpara{y,z}$ is a 2-cocycle of the induced 3-Lie algebra.
\end{theorem}
\begin{proof}
Let $\varphi \in Z^2_{ad}(A,A)$ and $\omega : A \to \mathbb{K}$ a linear map satisfying conditions 1,2 and 3 above, and let $\psi\AKMSpara{x,y,z} = \underset{x,y,z}{\circlearrowleft}\omega \AKMSpara{x} \varphi \AKMSpara{y,z}$, then we have:

\begin{align*}
d^2 \psi&\AKMSpara{x_1,x_2,y_1,y_2,z} = \psi \AKMSpara{x_1,x_2,\AKMSbracket{y_1,y_2,z}_\tau} - \psi \AKMSpara{\AKMSbracket{x_1,x_2,y_1}_\tau,y_2,z} \\
 &- \psi\AKMSpara{y_1,\AKMSbracket{x_1,x_2,y_2}_\tau,z} - \psi\AKMSpara{y_1,y_2,\AKMSbracket{x_1,x_2,z}_\tau} + \AKMSbracket{x_1,x_2,\psi\AKMSpara{y_1,y_2,z}}_\tau\\
& - \AKMSbracket{\psi\AKMSpara{x_1,x_2,y_1},y_2,z}_\tau - \AKMSbracket{y_1,\psi\AKMSpara{x_1,x_2,y_2},z}_\tau - \AKMSbracket{y_1,y_2,\psi\AKMSpara{x_1,x_2,z}}_\tau\\
&=\tau\AKMSpara{y_1}\psi\AKMSpara{x_1,x_2,\AKMSbracket{y_2,z}} + \tau\AKMSpara{y_2}\psi\AKMSpara{x_1,x_2,\AKMSbracket{z,y_1}} + \tau\AKMSpara{z}\psi\AKMSpara{x_1,x_2,\AKMSbracket{y_1,y_2}} \\
& - \tau\AKMSpara{x_1}\psi\AKMSpara{y_1,y_2,\AKMSbracket{x_2,z}} - \tau\AKMSpara{x_2}\psi\AKMSpara{y_1,y_2,\AKMSbracket{z,x_1}} - \tau\AKMSpara{z}\psi\AKMSpara{y_1,y_2,\AKMSbracket{x_1,x_2}} \\
& - \tau\AKMSpara{x_1}\psi\AKMSpara{\AKMSbracket{x_2,y_1}y_2,z,} - \tau\AKMSpara{x_2}\psi\AKMSpara{\AKMSbracket{y_1,x_1}y_2,z,} - \tau\AKMSpara{y_1}\psi\AKMSpara{\AKMSbracket{x_1,x_2}y_2,z,} \\
& - \tau\AKMSpara{x_1}\psi\AKMSpara{y_1,\AKMSbracket{x_2,y_2},z} - \tau\AKMSpara{x_2}\psi\AKMSpara{y_1,\AKMSbracket{y_2,x_1},z} - \tau\AKMSpara{y_2}\psi\AKMSpara{y_1,\AKMSbracket{x_1,x_2},z} \\
& + \tau\AKMSpara{x_1}\AKMSbracket{x_2,\psi\AKMSpara{y_1,y_2,z}} + \tau\AKMSpara{x_2}\AKMSbracket{\psi\AKMSpara{y_1,y_2,z},x_1} + \tau\AKMSpara{\psi\AKMSpara{y_1,y_2,z}}\AKMSbracket{x_1,x_2} \\
& - \tau\AKMSpara{y_1}\AKMSbracket{y_2,\psi\AKMSpara{x_1,x_2,z}} - \tau\AKMSpara{y_2}\AKMSbracket{\psi\AKMSpara{x_1,x_2,z},y_1} - \tau\AKMSpara{\psi\AKMSpara{x_1,x_2,z}}\AKMSbracket{y_1,y_2}\\
& - \tau\AKMSpara{\psi\AKMSpara{x_1,x_2,y_1}}\AKMSbracket{y_2,z} - \tau\AKMSpara{y_2}\AKMSbracket{z,\psi\AKMSpara{x_1,x_2,y_1}} - \tau\AKMSpara{z}\AKMSbracket{\psi\AKMSpara{x_1,x_2,y_1},y_2} \\
& - \tau\AKMSpara{y_1}\AKMSbracket{\psi\AKMSpara{x_1,x_2,y_2},z} - \tau\AKMSpara{\psi\AKMSpara{x_1,x_2,y_2}}\AKMSbracket{z,y_1} - \tau\AKMSpara{z}\AKMSbracket{y_1,\psi\AKMSpara{x_1,x_2,y_2}}\\
&= \tau\AKMSpara{y_1} \Big( \omega\AKMSpara{x_1} \varphi\AKMSpara{x_2,\AKMSbracket{y_2,z}} + \omega\AKMSpara{x_2} \varphi\AKMSpara{\AKMSbracket{y_2,z},a} - \omega\AKMSpara{y_2}\varphi\AKMSpara{z,\AKMSbracket{x_1,x_2}}\Big. \\ 
&\left. - \omega\AKMSpara{z}\varphi\AKMSpara{\AKMSbracket{x_1,x_2},y_2} - \omega\AKMSpara{x_1}\AKMSbracket{y_2,\varphi\AKMSpara{x_2,z}} - \omega\AKMSpara{x_2}\AKMSbracket{y_2,\varphi\AKMSpara{z,x_1}} \right.\\
&\left. - \omega\AKMSpara{z}\AKMSbracket{y_2,\varphi\AKMSpara{x_1,x_2}} - \omega\AKMSpara{x_1}\AKMSbracket{\varphi\AKMSpara{x_2,y_2},z} - \omega\AKMSpara{x_2}\AKMSbracket{\varphi\AKMSpara{y_2,x_1},z} \right. \\
& \Big. - \omega\AKMSpara{y}\AKMSbracket{\varphi\AKMSpara{x_1,x_2},z} \Big)\\
&+ \tau\AKMSpara{y_2}\Big( \omega\AKMSpara{x_1}\varphi\AKMSpara{x_2,\AKMSbracket{z,y_1} }+ \omega\AKMSpara{x_2}\varphi\AKMSpara{\AKMSbracket{z,y_1},x_1} - \omega\AKMSpara{x_1}\AKMSbracket{\varphi\AKMSpara{x_2,z},y_1} \Big.\\
&\left. - \omega\AKMSpara{x_2}\AKMSbracket{\varphi\AKMSpara{z,x_1},y_1} - \omega\AKMSpara{z}\AKMSbracket{\varphi\AKMSpara{x_1,x_2},y_1} - \omega\AKMSpara{x_1}\AKMSbracket{z,\varphi\AKMSpara{x_2,y_1}} \right.\\
& \left. - \omega\AKMSpara{x_2}\AKMSbracket{z,\varphi\AKMSpara{y_1,x_1}} - \omega\AKMSpara{y_1}\AKMSbracket{z,\varphi\AKMSpara{x_1,x_2}} - \omega\AKMSpara{y_1}\varphi\AKMSpara{\AKMSbracket{x_1,x_2},z}\right.\\
& \Big. - \omega\AKMSpara{z}\varphi\AKMSpara{y_1,\AKMSbracket{x_1,x_2}} \Big)\\
&+\tau\AKMSpara{z}\Big( \omega\AKMSpara{x_1}\varphi\AKMSpara{x_2,\AKMSbracket{y_1,y_2}} + \omega\AKMSpara{x_2}\varphi\AKMSpara{\AKMSbracket{y_1,y_2},x_1} - \omega\AKMSpara{y_1}\varphi\AKMSpara{y_2,\AKMSbracket{x_1,x_2}} \Big. \\
&\left. - \omega\AKMSpara{y_2}\varphi\AKMSpara{\AKMSbracket{x_1,x_2},y_1} - \omega\AKMSpara{x_1}\AKMSbracket{\varphi\AKMSpara{x_2,y_1},y_2} - \omega\AKMSpara{x_2}\AKMSbracket{\varphi\AKMSpara{y_1,x_1},y_2}\right.\\
&\left. - \omega\AKMSpara{y_1}\AKMSbracket{\varphi\AKMSpara{x_1,x_2},y_2} - \omega\AKMSpara{x_1}\AKMSbracket{y_1,\varphi\AKMSpara{x_2,y_2}} - \omega\AKMSpara{x_2}\AKMSbracket{y_1,\varphi\AKMSpara{y_2,x_1}} \right.\\
& \Big. -\omega\AKMSpara{y_2}\AKMSbracket{y_1,\varphi\AKMSpara{x_1,x_2}} \Big)\\
& + \tau\AKMSpara{x_1} \Big( \omega\AKMSpara{y_1}\AKMSbracket{x_2,\varphi\AKMSpara{y_2,z}} + \omega\AKMSpara{y_2}\AKMSbracket{x_2,\varphi\AKMSpara{z,y_1}} + \omega\AKMSpara{z}\AKMSbracket{x_2,\varphi\AKMSpara{y_1,y_2}} \Big.\\
& \left. - \omega\AKMSpara{y_1}\varphi\AKMSpara{y_2,\AKMSbracket{x_2,z}} - \omega\AKMSpara{y_2}\varphi\AKMSpara{\AKMSbracket{x_2,z},y_1} - \omega\AKMSpara{y_2}\varphi\AKMSpara{z,\AKMSbracket{x_2,y_1}} \right.\\
& \Big. - \omega\AKMSpara{z}\varphi\AKMSpara{\AKMSbracket{x_2,y_1},y_2} - \omega\AKMSpara{y_1}\varphi\AKMSpara{\AKMSbracket{x_2,y_2},z} - \omega\AKMSpara{z}\varphi\AKMSpara{y_1,\AKMSbracket{x_2,y_2}} \Big)\\
& + \tau\AKMSpara{x_2}\Big( \omega\AKMSpara{y_1}\AKMSbracket{\varphi\AKMSpara{y_2,z},x_1} + \omega\AKMSpara{y_2}\AKMSbracket{\varphi\AKMSpara{z,y_1},x_1} + \omega\AKMSpara{z}\AKMSbracket{\varphi\AKMSpara{y_1,y_2},x_1} \Big.\\
& \left. - \omega\AKMSpara{y_1}\varphi\AKMSpara{y_2,\AKMSbracket{z,x_1}} - \omega\AKMSpara{y_2}\varphi\AKMSpara{\AKMSbracket{z,x_1},y_1} - \omega\AKMSpara{y_2}\varphi\AKMSpara{z,\AKMSbracket{y_1,x_1}} \right.\\
& \Big. - \omega\AKMSpara{z}\varphi\AKMSpara{\AKMSbracket{y_1,x_1},y_2} - \omega\AKMSpara{y_1}\varphi\AKMSpara{\AKMSbracket{y_2,x_1},z} - \omega\AKMSpara{z}\varphi\AKMSpara{y_1,\AKMSbracket{y_2,x_1}} \Big)\\
&+ \Big( \omega\AKMSpara{y_1}\tau\AKMSpara{\varphi\AKMSpara{y_2,z}} + \omega\AKMSpara{y_2}\tau\AKMSpara{\varphi\AKMSpara{z,y_1}} + \omega\AKMSpara{z}\tau\AKMSpara{\varphi\AKMSpara{y_1,y_2}} \Big) \AKMSbracket{x_1,x_2} \\
& - \Big( \omega\AKMSpara{x_1}\tau\AKMSpara{\varphi\AKMSpara{x_2,y_1}} + \omega\AKMSpara{x_2}\tau\AKMSpara{\varphi\AKMSpara{y_1,x_1}} + \omega\AKMSpara{y_1}\tau\AKMSpara{\varphi\AKMSpara{x_1,x_2}} \Big) \AKMSbracket{y_2,z} \\
& - \Big( \omega\AKMSpara{x_1}\tau\AKMSpara{\varphi\AKMSpara{x_2,z}} + \omega\AKMSpara{x_2}\tau\AKMSpara{\varphi\AKMSpara{z,x_1}} + \omega\AKMSpara{z}\tau\AKMSpara{\varphi\AKMSpara{x_1,x_2}} \Big) \AKMSbracket{y_1,y_2} \\
& - \Big( \omega\AKMSpara{x_1}\tau\AKMSpara{\varphi\AKMSpara{x_2,y_2}} + \omega\AKMSpara{x_2}\tau\AKMSpara{\varphi\AKMSpara{y_2,x_1}} + \omega\AKMSpara{y_2}\tau\AKMSpara{\varphi\AKMSpara{x_1,x_2}} \Big) \AKMSbracket{z,y_1} \\
&= - \tau\AKMSpara{y_1}\omega\AKMSpara{x_1}\delta^2\varphi\AKMSpara{z,y_2,x_2} - \tau\AKMSpara{y_1}\omega\AKMSpara{x_2}\delta^2\varphi\AKMSpara{y_2,z,x_1} \\
 &-\tau\AKMSpara{y_2}\omega\AKMSpara{x_1}\delta^2\varphi\AKMSpara{y_1,z,x_2} - \tau\AKMSpara{y_2}\omega\AKMSpara{x_2}\delta^2\varphi\AKMSpara{z,y_1,x_1} \\
& -\tau\AKMSpara{z}\omega\AKMSpara{x_1}\delta^2\varphi\AKMSpara{y_2,y_1,x_2} - \tau\AKMSpara{z}\omega\AKMSpara{x_2}\delta^2\varphi\AKMSpara{y_1,y_2,x_1}\\
&+ \Big( \omega\AKMSpara{y_1}\tau\AKMSpara{\varphi\AKMSpara{y_2,z}} + \omega\AKMSpara{y_2}\tau\AKMSpara{\varphi\AKMSpara{z,y_1}} + \omega\AKMSpara{z}\tau\AKMSpara{\varphi\AKMSpara{y_1,y_2}} \Big) \AKMSbracket{x_1,x_2} \\
& - \Big( \omega\AKMSpara{x_1}\tau\AKMSpara{\varphi\AKMSpara{x_2,y_1}} + \omega\AKMSpara{x_2}\tau\AKMSpara{\varphi\AKMSpara{y_1,x_1}} + \omega\AKMSpara{y_1}\tau\AKMSpara{\varphi\AKMSpara{x_1,x_2}} \Big) \AKMSbracket{y_2,z} \\
& - \Big( \omega\AKMSpara{x_1}\tau\AKMSpara{\varphi\AKMSpara{x_2,z}} + \omega\AKMSpara{x_2}\tau\AKMSpara{\varphi\AKMSpara{z,x_1}} + \omega\AKMSpara{z}\tau\AKMSpara{\varphi\AKMSpara{x_1,x_2}} \Big) \AKMSbracket{y_1,y_2} \\
& - \Big( \omega\AKMSpara{x_1}\tau\AKMSpara{\varphi\AKMSpara{x_2,y_2}} + \omega\AKMSpara{x_2}\tau\AKMSpara{\varphi\AKMSpara{y_2,x_1}} + \omega\AKMSpara{y_2}\tau\AKMSpara{\varphi\AKMSpara{x_1,x_2}} \Big) \AKMSbracket{z,y_1}. \\
\end{align*}
Since \[\underset{x,y,z}{\LARGE{\circlearrowleft }}\omega \AKMSpara{x} \tau \AKMSpara{ \varphi \AKMSpara{ y,z} } = 0, \forall x,y,z \in A, \]
we get \[d^2 \psi = 0.\]
\qed
\end{proof}

\begin{theorem}
Every $1$-cocycle for the scalar cohomology of $(A,\AKMSbracket{.,.})$ is a $1$-cocycle for the scalar cohomology of the induced $3$-Lie algebra.
\end{theorem}
\begin{proof}
Let $\omega$ be a $1$-cocycle for the scalar cohomology of $(A,\AKMSbracket{.,.})$, then 
\[\forall x,y \in A, \delta^1 \omega(x,y) = \omega \AKMSpara{\AKMSbracket{x,y}} = 0,\]
 which is equivalent to $\AKMSbracket{A,A} \subset \ker \omega$. By Remark \ref{AKMS-D3subD} $\AKMSbracket{A,A,A}_\tau \subset \AKMSbracket{A,A}$ and then $\AKMSbracket{A,A,A}_\tau \subset \ker \omega$, that is 
\[ \forall x,y,z \in A, \omega\AKMSpara{\AKMSbracket{x,y,z}_\tau}=d^1 \omega \AKMSpara{x,y,z} = 0,\]
 which means that $\omega$ is a $1$-cocycle for the scalar cohomology of $\AKMSpara{A,\AKMSbracket{.,.,.}_\tau}$.  \qed
\end{proof}

\begin{theorem}\label{AKMS-Z2tri}
Let $\varphi \in Z^2_{0}(A,\mathbb{K})$ and $\omega : A \to \mathbb{K}$ a linear map satisfying :
\begin{enumerate}
\item $\tau(x)\omega(y) = \tau(y)\omega(x)$,
\item $\omega([x,y])=0$,
\item $\omega(y_2)\AKMSpara{\tau(x_1)\varphi\AKMSpara{\AKMSbracket{x_1,z}x_2}+\tau(x_2)\varphi\AKMSpara{\AKMSbracket{z,y_1}x_1}}=0$.
\end{enumerate}
Then $\psi\AKMSpara{x,y,z} = \underset{x,y,z}{\LARGE{\circlearrowleft }}\omega \AKMSpara{x} \varphi \AKMSpara{y,z}$ is a $2$-cocycle of the induced $3$-Lie algebra.
\end{theorem}
\begin{proof}
Let $\varphi \in Z^2_{0}(A,\mathbb{K})$ and $\omega : A \to \mathbb{K}$ a linear map satisfying conditions 1, 2 and 3 above, and let $\psi\AKMSpara{x,y,z} = \underset{x,y,z}{\LARGE{\circlearrowleft }}\omega \AKMSpara{x} \varphi \AKMSpara{y,z}$, then we have:
\begin{align*}
d^2 \psi &\AKMSpara{x_1,x_2,y_1,y_2,z} = \psi \AKMSpara{x_1,x_2,\AKMSbracket{y_1,y_2,z}_\tau} - \psi \AKMSpara{\AKMSbracket{x_1,x_2,y_1}_\tau,y_2,z} \\
&- \psi\AKMSpara{y_1,\AKMSbracket{x_1,x_2,y_2}_\tau,z} - \psi\AKMSpara{y_1,y_2,\AKMSbracket{x_1,x_2,z}_\tau}\\
&= \tau\AKMSpara{y_1}\psi\AKMSpara{x_1,x_2,\AKMSbracket{y_2,z}}+\tau\AKMSpara{y_2}\psi\AKMSpara{x_1,x_2,\AKMSbracket{z,y_1}}+\tau\AKMSpara{z}\psi\AKMSpara{x_1,x_2,\AKMSbracket{y_1,y_2}}\\
& - \tau\AKMSpara{x_1}\psi\AKMSpara{\AKMSbracket{x_2,y_1},y_2,z} - \tau\AKMSpara{x_2}\psi\AKMSpara{\AKMSbracket{y_1,x_1},y_2,z} - \tau\AKMSpara{y_1}\psi\AKMSpara{\AKMSbracket{x_1,x_2},y_2,z}\\
& - \tau\AKMSpara{x_1}\psi\AKMSpara{y_1,\AKMSbracket{x_2,y_2},z} - \tau\AKMSpara{x_2}\psi\AKMSpara{y_1,\AKMSbracket{y_2,x_1},z} - \tau\AKMSpara{y_2}\psi\AKMSpara{y_1,\AKMSbracket{x_1,x_2},z}\\
& - \tau\AKMSpara{x_1}\psi\AKMSpara{y_1,y_2,\AKMSbracket{x_2,z}} - \tau\AKMSpara{x_2}\psi\AKMSpara{y_1,y_2,\AKMSbracket{z,x_1}} - \tau\AKMSpara{z}\psi\AKMSpara{y_1,y_2,\AKMSbracket{x_1,x_2}}\\
&= \tau\AKMSpara{y_1}\AKMSpara{\omega\AKMSpara{x_1}\varphi\AKMSpara{x_2,\AKMSbracket{y_2,z}} + \omega\AKMSpara{x_2}\varphi\AKMSpara{\AKMSbracket{y_2,z},x_1}} \\
& + \tau\AKMSpara{y_2}\AKMSpara{\omega\AKMSpara{x_1}\varphi\AKMSpara{x_2,\AKMSbracket{y_1,z}} + \omega\AKMSpara{x_2}\varphi\AKMSpara{\AKMSbracket{y_1,z},x_1}} \\
& + \tau\AKMSpara{z}\AKMSpara{\omega\AKMSpara{x_1}\varphi\AKMSpara{x_2,\AKMSbracket{y_1,y_2}} + \omega\AKMSpara{x_2}\varphi\AKMSpara{\AKMSbracket{y_1,y_2},x_1}}\\
& - \tau\AKMSpara{x_1}\AKMSpara{\omega\AKMSpara{y_2}\varphi\AKMSpara{z,\AKMSbracket{x_2,y_1}} + \omega\AKMSpara{z}\varphi\AKMSpara{\AKMSbracket{x_2,y_1},y_2}}\\
& - \tau\AKMSpara{x_2}\AKMSpara{\omega\AKMSpara{y_2}\varphi\AKMSpara{z,\AKMSbracket{y_1,x_1}} + \omega\AKMSpara{z}\varphi\AKMSpara{\AKMSbracket{y_1,x_1}y_2}}\\
& - \tau\AKMSpara{y_1}\AKMSpara{\omega\AKMSpara{y_2}\varphi\AKMSpara{z,\AKMSbracket{x_1,x_2}} + \omega{z}\varphi\AKMSpara{\AKMSbracket{x_1,x_2},y_2}}\\
& - \tau\AKMSpara{x_1}\AKMSpara{\omega\AKMSpara{y_1}\varphi\AKMSpara{\AKMSbracket{x_2,y_2},z} + \omega\AKMSpara{z}\varphi\AKMSpara{y_1,\AKMSbracket{x_2,y_2}}} \\
& - \tau\AKMSpara{x_2}\AKMSpara{\omega\AKMSpara{y_1}\varphi\AKMSpara{\AKMSbracket{y_2,x_1},z} + \omega\AKMSpara{z}\varphi\AKMSpara{y_1,\AKMSbracket{y_2,x_1}}} \\
& - \tau\AKMSpara{y_2}\AKMSpara{\omega\AKMSpara{y_1} \varphi\AKMSpara{\AKMSbracket{x_1,x_2},z} + \omega\AKMSpara{z}\varphi\AKMSpara{y_1,\AKMSbracket{x_1,x_2}}} \\
& - \tau\AKMSpara{x_1}\AKMSpara{\omega\AKMSpara{y_1}\varphi\AKMSpara{y_2,\AKMSbracket{x_2,z}} + \omega\AKMSpara{y_2}\varphi\AKMSpara{\AKMSbracket{x_2,z},y_1}} \\
& - \tau\AKMSpara{x_2}\AKMSpara{\omega\AKMSpara{y_1}\varphi\AKMSpara{y_2,\AKMSbracket{z,x_1}} + \omega\AKMSpara{y_2}\varphi\AKMSpara{\AKMSbracket{z,x_1},y_1}} \\
& - \tau\AKMSpara{z}\AKMSpara{\omega\AKMSpara{y_1}\varphi\AKMSpara{y_2,\AKMSbracket{x_1,x_2}} + \omega\AKMSpara{y_2}\varphi\AKMSpara{\AKMSbracket{x_1,x_2},y_1}}\\
& = \tau\AKMSpara{x_1}\omega\AKMSpara{y_1}\delta^2\varphi\AKMSpara{y_2,z,x_2} + \tau\AKMSpara{x_1}\omega\AKMSpara{y_2}\delta^2\varphi\AKMSpara{z,y_1,x_2}\\
& + \tau\AKMSpara{x_2}\omega\AKMSpara{y_1}\delta^2\varphi\AKMSpara{z,y_2,x_1} + \tau\AKMSpara{x_2}\omega\AKMSpara{y_2}\delta^2\varphi\AKMSpara{x_1,y_1,z}\\
& + \tau\AKMSpara{x_1}\omega\AKMSpara{z}\delta^2\varphi\AKMSpara{y_1,y_2,x_2} + \tau\AKMSpara{x_2}\omega\AKMSpara{z}\delta^2\varphi\AKMSpara{y_2,y_1,x_1}\\
& -2 \omega\AKMSpara{y_2}\AKMSpara{\tau\AKMSpara{x_1}\varphi\AKMSpara{\AKMSbracket{y_1,z},x_2} + \tau\AKMSpara{x_2}\varphi\AKMSpara{\AKMSbracket{z,y_1},x_1}}. 
\end{align*}
Since \[  \omega\AKMSpara{y_2}\AKMSpara{\tau\AKMSpara{x_1}\varphi\AKMSpara{\AKMSbracket{y_1,z},x_2} + \tau\AKMSpara{x_2}\varphi\AKMSpara{\AKMSbracket{z,y_1},x_1}} = 0, \] it follows that $d^2\psi = 0$.
\qed
\end{proof}

\begin{remark}
Condition 1 in Theorems \ref{AKMS-Z2ad} and \ref{AKMS-Z2tri} are equivalent to $\omega = \lambda \tau, \lambda \in \mathbb{K}$, and therefore one may remove condition 2, which is redundant.
\end{remark}
\begin{lemma}\label{AKMScoBtau}
Let $\alpha \in C^1(A,\mathbb{K})$. Then:
\[ d^1 \alpha\AKMSpara{x,y,z} = \underset{x,y,z}{\circlearrowleft }\tau \AKMSpara{x} \delta^1 \alpha\AKMSpara{y,z}, \forall x,y,z \in A. \]
\end{lemma}
\begin{proof}
Let $\alpha \in C^1(A,\mathbb{K})$, $x,y,z \in A$, then we have:
\[ d^1 \alpha\AKMSpara{x,y,z} = \alpha\AKMSpara{\AKMSbracket{x,y,z}}=\underset{x,y,z}{\circlearrowleft }\tau \AKMSpara{x} \alpha\AKMSpara{\AKMSbracket{y,z}}= \underset{x,y,z}{\circlearrowleft }\tau \AKMSpara{x} \delta^1 \alpha\AKMSpara{y,z}\]\qed
\end{proof}

\begin{proposition}\label{AKMSB2triv_eq}
Let $\varphi_1,\varphi_2 \in Z^2_{0}(A,\mathbb{K})$ satisfying conditions of Theorem \ref{AKMS-Z2tri}. If $\varphi_1,\varphi_2$ are in the same cohomology class then $\psi_1,\psi_2$ defined by:
\[ \psi_i\AKMSpara{x,y,z} = \underset{x,y,z}{\circlearrowleft }\tau \AKMSpara{x} \varphi_i \AKMSpara{y,z}, i=1,2 \]
are in the same cohomology class.
\end{proposition}
\begin{proof}
Let $\varphi_1,\varphi_2 \in Z^2_{0}(A,\mathbb{K})$ be two cocycles in the same cohomology class, that is \[\varphi_2 - \varphi_1 = \delta^1\alpha, \alpha \in C^1(A,\mathbb{K})\] satisfying conditions of Theorem \ref{AKMS-Z2tri}, and 
\[ \psi_i\AKMSpara{x,y,z} = \underset{x,y,z}{\circlearrowleft }\tau \AKMSpara{x} \varphi_i \AKMSpara{y,z} : i=1,2, \]
then we have:
\begin{align*}
\psi_2\AKMSpara{x,y,z}-\psi_1\AKMSpara{x,y,z} &= \underset{x,y,z}{\circlearrowleft }\tau \AKMSpara{x} \varphi_2 \AKMSpara{y,z} - \underset{x,y,z}{\circlearrowleft }\tau \AKMSpara{x} \varphi_1 \AKMSpara{y,z}\\
&= \underset{x,y,z}{\circlearrowleft }\tau \AKMSpara{x} \AKMSpara{\varphi_2-\varphi_1} \AKMSpara{y,z}\\
&= \underset{x,y,z}{\circlearrowleft }\tau \AKMSpara{x} \delta^1 \alpha \AKMSpara{y,z}\\
&= d^1 \alpha\AKMSpara{x,y,z}.
\end{align*}
Which means that $\psi_1$ and $\psi_2$ are in the same cohomology class. \qed
\end{proof}

%%%%%%%%%%%%%%%%%%%%%%%%%%%%%%%%%%%%%%%%%%%%%%%%%%%%%%

%%%%%%%%%%%%%%%%%%%%%%%%%%%%%%%%%%%%%%%%%%%%%%%%%%%%%%

\section{Central Extension of $3$-Lie Algebras Induced by Lie Algebras}\label{AKMS:Extension}
\begin{definition}
Let $A,B,C$ be $n$-Lie algebras ($n\geq 2$). An extension of $B$ by $A$ is a short sequence: 
\[ A \overset{\lambda}{\to} C \overset{\mu}{\to} B, \]
such that $\lambda$ is an injective homomorphism, $\mu$ is a surjective homomorphism, and $\operatorname{Im} \lambda \subset \ker \mu$. We say also that $C$ is an extension of $B$ by $A$.
\end{definition}
\begin{definition}
Let $A$, $B$ be $n$-Lie algebras, and  $A \overset{\lambda}{\to} C \overset{\mu}{\to} B$ be an extension of $B$ by $A$.
\begin{itemize}
\item The extension is said to be trivial if there exists an ideal $I$ of $C$ such that $C = \ker \mu \oplus I$.
\item It is said to be central if $\ker \mu \subset Z (C)$.
\end{itemize}
\end{definition}
We may equivalently define central extensions by a $1$-dimensional algebra (we will simply call it central extension) this way:
\begin{definition}
Let $A$ be a $n$-Lie algebra, we call central extension of $A$ the space $\bar{A}=A\oplus \mathbb{K} c$ equipped with the bracket:
\[ \forall x_1,...,x_n \in A, \AKMSbracket{x_1,...,x_n}_c = \AKMSbracket{x_1,...,x_n} + \omega\AKMSpara{x_1,...,x_n} c \text{ and } \AKMSbracket{x_1,...,x_{n-1},c}_c = 0. \] 
Where $\omega$ is a skew-symmetric $n$-linear form such that $\AKMSbracket{\cdot,...,\cdot}_c$ satisfies the fundamental identity (or Jacobi identity for $n=2$).
\end{definition}

\begin{proposition}[\cite{aip:review}]
\begin{enumerate}
\item The bracket of a central extension satisfies the fundamental identity (resp. Jacobi identity) if and only if $\omega$ is a $2$-cocycle for the scalar cohomology of $n$-Lie algebras (resp. Lie algebras).
\item Two central extensions of a $n$-Lie algebra (resp. Lie algebra) $A$ given by two maps $\omega_1$ and $\omega_2$ are isomorphic if and only if $\omega_2 - \omega_1$ is a $2$-coboundary for the scalar cohomology of $n$-Lie algebras (resp. Lie algebras).
\end{enumerate}
\end{proposition}
Now, we look at the question of whether a central extension of a Lie algebra may give a central extension of the induced $3$-Lie algebra (by some trace $\tau$), the answer is given by the following theorem:
\begin{theorem}
Let $(A,\AKMSbracket{.,.})$ be a Lie algebra, $\tau$ be a trace and $\AKMSpara{A,\AKMSbracket{.,.,.}_\tau}$ be the induced $3$-Lie algebra. If $\AKMSpara{\bar{A},\AKMSbracket{.,.}_c}$ is a central extension of $(A,\AKMSbracket{.,.})$ where \[\bar{A}=A\oplus \mathbb{K} c \text{ and } \AKMSbracket{x,y}_c = \AKMSbracket{x,y} + \omega\AKMSpara{x,y}c, \] and we extend $\tau$ to $\bar{A}$ by assuming $\tau(c)=0$ then $\AKMSpara{\bar{A},\AKMSbracket{.,.,.}_{c,\tau}}$ the $3$-Lie algebra induced by $\AKMSpara{\bar{A},\AKMSbracket{.,.}_c}$, is a central extension of $(A,\AKMSbracket{.,.,.}_\tau)$.
\end{theorem} 
\begin{proof}
Let $x,y,z \in A$:
\begin{align*}
\AKMSbracket{x,y,z}_{c,\tau} &= \tau\AKMSpara{x}\AKMSbracket{y,z}_c + \tau\AKMSpara{y}\AKMSbracket{z,x}_c + \tau\AKMSpara{z}\AKMSbracket{x,y}_c \\
&= \tau\AKMSpara{x}\AKMSpara{\AKMSbracket{y,z} + \omega\AKMSpara{y,z}c} + \tau\AKMSpara{y}\AKMSpara{\AKMSbracket{z,x} + \omega\AKMSpara{z,x} c} + \tau\AKMSpara{z}\AKMSpara{\AKMSbracket{x,y} + \omega\AKMSpara{x,y} c} \\
&= \AKMSpara{\tau\AKMSpara{x}\AKMSbracket{y,z} + \tau\AKMSpara{y}\AKMSbracket{z,x} + \tau\AKMSpara{z}\AKMSbracket{x,y}} \\
&+ \AKMSpara{\tau\AKMSpara{x}\omega\AKMSpara{y,z} + \tau\AKMSpara{y}\omega\AKMSpara{z,x} + \tau\AKMSpara{z}\omega\AKMSpara{x,y}} c. \\
&= \AKMSbracket{x,y,z}_\tau + \omega_\tau \AKMSpara{x,y,z} c
\end{align*}
The map $\omega_\tau \AKMSpara{x,y,z} = \tau\AKMSpara{x}\omega\AKMSpara{y,z} + \tau\AKMSpara{y}\omega\AKMSpara{z,x} + \tau\AKMSpara{z}\omega\AKMSpara{x,y}$ is a skew-symmetric $3$-linear form, and $\AKMSbracket{.,.,.}_{c,\tau}$ satisfies the fundamental identity, we have also:
\begin{align*}
\AKMSbracket{x,y,c}_{c,\tau}&= \tau\AKMSpara{x}\AKMSbracket{y,c}_c + \tau\AKMSpara{y}\AKMSbracket{c,x}_c + \tau\AKMSpara{c}\AKMSbracket{x,y}_c \\
&= 0. \qquad \Big(\AKMSbracket{y,c}_c = \AKMSbracket{c,x}_c = 0 \text{ and } \tau\AKMSpara{c} = 0.\Big)
\end{align*}
Therefore $\AKMSpara{\bar{A},\AKMSbracket{.,.,.}_{c,\tau}}$ is a central extension of $(A,\AKMSbracket{.,.,.}_\tau)$. \qed
\end{proof}

\begin{example}
Consider the $4$-dimensional Lie algebra $(A,\AKMSbracket{.,.})$ with basis $\{e_1,e_2,e_3,e_4\}$ defined by:
\[ \AKMSbracket{e_2,e_4}=e_3\ ;\ \AKMSbracket{e_3,e_4}=e_3, \]
(remaining brackets are either obtained by skew-symmetry or zero), and let $\omega$ be a skew-symmetric bilinear form on $A$. $\omega$ is fully defined by the scalars 
\[\omega_{ij}=\omega\AKMSpara{e_i,e_j}, 1\leq i<j \leq 4.\]
By solving the equations for $\omega$ to be a $2$-cocycle: \[\delta^2\omega\AKMSpara{e_i,e_j,e_k}=0, 1\leq i<j<k\leq 4,\]  we get the conditions:
\[ \omega_{13}=0 \text{ and } \omega_{23}=0. \]
Now, let $\alpha$ be a linear form on $A$, defined by $\alpha\AKMSpara{e_i}=\alpha_i, 1\leq i \leq 4$, we find that $\delta^1\alpha\AKMSpara{e_2,e_4}=\delta^1\alpha\AKMSpara{e_3,e_4}=\alpha_3$ and $\delta^1\alpha\AKMSpara{e_i,e_j}=0$ for other values of $i$ and $j$ ($i<j$).
Now consider the trace map $\tau$ such that $\tau\AKMSpara{e_1}=1$ and $\tau\AKMSpara{e_i}=0, i \neq 1$, and the $2$-cocycles $\lambda$ and $\mu$ defined by:
\[ \lambda\AKMSpara{e_1,e_2}=1 \]
and
\[ \mu\AKMSpara{e_2,e_4}=1\ ;\  \mu\AKMSpara{e_3,e_4}=-1. \]
Central extensions of $(A,\AKMSbracket{.,.})$ by $\lambda$ and $\mu$ are respectively given by ($\bar{A} = A \oplus \mathbb{K} c$):
\[ \AKMSbracket{e_1,e_2}_\lambda = c\ ;\ \AKMSbracket{e_2,e_4}_\lambda = e_3\ ;\  \AKMSbracket{e_3,e_4}_\lambda = e_3 \]
and 
\[ \AKMSbracket{e_2,e_4}_\mu = e_3 + c\ ;\ \AKMSbracket{e_3,e_4}_\mu = e_3 - c. \]

$3$-Lie algebras induced by $(A,\AKMSbracket{.,.})$ and by these central extensions are given by:
\[ \AKMSbracket{e_1,e_2,e_4}_\tau = e_3\ ;\ \AKMSbracket{e_1,e_3,e_4}_\tau = e_3, \]
\[ \AKMSbracket{e_1,e_2,e_4}_{\tau,\lambda} = e_3\ ;\ \AKMSbracket{e_1,e_3,e_4}_{\tau,\lambda} = e_3 \] and
\[ \AKMSbracket{e_1,e_2,e_4}_{\tau,\mu} = e_3 + c\ ;\ \AKMSbracket{e_1,e_3,e_4}_{\tau,\mu} = e_3 - c. \]
We can see that, here, the central extension given by $\lambda$ induces a trivial one, while the one given by $\mu$ induces a non-trivial one. This example shows also that the converse of Proposition \ref{AKMSB2triv_eq} is, in general, not true.
\end{example}

%%%%%%%%%%%%%%%%%%%%%%%%%%%%%%%%%%%%%%%%%%%%%%%%%%%%%%

\section{$3$-Lie Algebras Induced by Lie Algebras in Low Dimensions}\label{AKMS:Classifications}
In this section, we give a list of all $3$-Lie algebras induced by Lie algebras in dimension $d\leq 5$, based on the classifications given in \cite{Filippov:nLie} and \cite{Bai:nLie:n+2}. For this, we shall use the following result:
\begin{proposition} \label{AKMS-3to2}
Let $\AKMSpara{A,\AKMSbracket{.,.,.}}$ be a $3$-Lie algebra, $\AKMSpara{e_i}_{1\leq i \leq d}$ a basis of $A$. If there exists $e_{i_0}$ in this base, such that the multiplication table of $\AKMSpara{A,\AKMSbracket{.,.,.}}$ is given by:
\[ \AKMSbracket{e_{i_0},e_j,e_k} = x_{jk} ; j\neq i_0, k\neq i_0, k\neq j \]
with $e_{i_0}$ and $x_{jk}$ linearly independent, then $\AKMSpara{A,\AKMSbracket{.,.,.}}$ is induced by a Lie algebra
\end{proposition}
\begin{proof}
We define a bilinear skew-symmetric map $\AKMSbracket{.,.}$ on $A$ and a form $\tau : A \to \mathbb{K}$ by:
\[ \AKMSbracket{e_j,e_k}=x_{jk}, j\neq i_0, k\neq i_0, k\neq j \text{ and } \AKMSbracket{e_{i_0},e_j} = 0 \]
and
\[ \tau(x) = \tau\AKMSpara{\sum_{k=0}^d x_k e_k} = x_{i_0} \]
$\AKMSbracket{.,.}$ satisfies the Jacobi identity:
\begin{align*}
\AKMSbracket{e_j,\AKMSbracket{e_k,e_l}} &= \AKMSbracket{e_{i_0},e_j,\AKMSbracket{e_{i_0},e_k,e_l}}\\
&= \AKMSbracket{\AKMSbracket{e_{i_0},e_j,e_{i_0}},e_k,e_l} + \AKMSbracket{e_{i_0},\AKMSbracket{e_{i_0},e_j,e_k},e_l} + \AKMSbracket{e_{i_0},e_k,\AKMSbracket{e_{i_0,}e_j,e_l}} \\
&= \AKMSbracket{\AKMSbracket{e_j,e_k},e_l}+\AKMSbracket{e_k,\AKMSbracket{e_j,e_l}}
\end{align*}
The obtained Lie bracket $\AKMSbracket{.,.}$ and the trace $\tau$ given above indeed induce the ternary bracket considered above:
\begin{align*}
\AKMSbracket{e_{i_0},e_j,e_k}_\tau &= \tau(e_{i_0})\AKMSbracket{e_j,e_k} + \tau(e_j)\AKMSbracket{e_k,e_{i_0}} + \tau(e_k)\AKMSbracket{e_{i_0},e_j}\\
&= \tau(e_{i_0})\AKMSbracket{e_j,e_k} \\
&= x_{jk}\\
&= \AKMSbracket{e_{i_0},e_j,e_k}
\end{align*}
for $i \neq i_0$:
\[ \AKMSbracket{e_i,e_j,e_k}_\tau = \tau(e_i)\AKMSbracket{e_j,e_k} + \tau(e_j)\AKMSbracket{e_k,e_i} + \tau(e_k)\AKMSbracket{e_i,e_j} = 0 = \AKMSbracket{e_i,e_j,e_k} \] 
\qed
\end{proof}

\begin{theorem}[\cite{Filippov:nLie} $3$-Lie algebras of dimension less than or equal to $4$] \label{AKMSd4}
Any $3$-Lie algebra $A$ of dimension less than or equal to $4$ is isomorphic to one of the following algebras: (omitted brackets are obtained by skew-symmetry, $\AKMSpara{e_i}_{1 \leq i \leq dim A}$ is a basis of $A$)
\begin{enumerate}
\item If $dim A < 3$ then $A$ is abelian.
\item If $dim A = 3$, then we have 2 cases :
\begin{enumerate}
\item $A$ is abelian.
\item $\AKMSbracket{e_1,e_2,e_3} = e_1.$
\end{enumerate}
\item if $dim A = 4$ then we have the following cases :
\begin{enumerate}
\item $A$ is abelian.
\item $\AKMSbracket{e_2,e_3,e_4} = e_1$.
\item $\AKMSbracket{e_1,e_2,e_3}=e_1$.
\item $\AKMSbracket{e_1,e_2,e_4} = a e_3 + b e_4 ; \AKMSbracket{e_1,e_2,e_3} = c e_3 + d e_4$, with $C=
\begin{pmatrix}
a & b \\
c & d
\end{pmatrix}$ an invertible matrix. Two such algebras, defined by matrices $C_1$ and $C_2$, are isomorphic if and only if there exists a scalar $\alpha$ and an invertible matrix $B$ such that $C_2 = \alpha B C_1 B^{-1}$.
\item $\AKMSbracket{e_2,e_3,e_4}= e_1 ; \AKMSbracket{e_1,e_3,e_4}=a e_2 ; \AKMSbracket{e_1,e_2,e_4}= b e_3$ ($a,b \neq 0$).
\item  $\AKMSbracket{e_2,e_3,e_4}= e_1 ; \AKMSbracket{e_1,e_3,e_4}=a e_2 ; \AKMSbracket{e_1,e_2,e_4}= b e_3 ; \AKMSbracket{e_1,e_2,e_3}= c e_4$ ($a,b,c \neq 0$).
\end{enumerate}

\end{enumerate}
\end{theorem}

\begin{theorem}[\cite{Bai:nLie:n+2} $5$-dimensional $3$-Lie algebras] \label{AKMSd5}
Let $\mathbb{K}$ be an algebraically closed field. Any $5$-dimensional $3$-Lie algebra $A$ defined with respect to a basis $\left\{e_1,e_2,e_3,e_4,e_5\right\}$ is isomorphic to one of the algebras listed bellow, where $A^1$ denotes $\AKMSbracket{A,A,A}$ :

\begin{enumerate}
\item If $dim A^1 = 0$ then $A$ is abelian.
\item If $dim A^1 = 1$, let $A^1=\langle e_1 \rangle$, then we have :
\begin{enumerate}
\item $A^1 \subseteq Z(A)$ : $\AKMSbracket{e_2,e_3,e_4}=e_1$.
\item $A^1 \nsubseteq Z(A)$ : $\AKMSbracket{e_1,e_2,e_3}=e_1$.
\end{enumerate}
\item If $dim A^1 = 2$, let $A^1 = \langle e_1,e_2 \rangle$, then we have :
\begin{enumerate}
\item $\AKMSbracket{e_2,e_3,e_4}=e_1 ; \AKMSbracket{e_3,e_4,e_5}=e_2$.
\item $\AKMSbracket{e_2,e_3,e_4}=e_1 ; \AKMSbracket{e_2,e_4,e_5}=e_2 ; \AKMSbracket{e_1,e_4,e_5}=e_1$.
\item $\AKMSbracket{e_2,e_3,e_4}=e_1 ; \AKMSbracket{e_1,e_3,e_4}=e_2$.
\item $\AKMSbracket{e_2,e_3,e_4}=e_1 ; \AKMSbracket{e_1,e_3,e_4}=e_2 ; \AKMSbracket{e_2,e_4,e_5}=e_2 ; \AKMSbracket{e_1,e_4,e_5}=e_1$.
\item $\AKMSbracket{e_2,e_3,e_4}= \alpha e_1 + e_2 ; \AKMSbracket{e_1,e_3,e_4}=e_2$.
\item $\AKMSbracket{e_2,e_3,e_4}=\alpha e_1+e_2 ; \AKMSbracket{e_1,e_3,e_4}=e_2 ; \AKMSbracket{e_2,e_4,e_5}=e_2 ; \AKMSbracket{e_1,e_4,e_5}=e_1$.
\item $\AKMSbracket{e_1,e_3,e_4}=e_1 ; \AKMSbracket{e_2,e_3,e_4}=e_2$.
\end{enumerate}
where $\alpha \in \mathbb{K} \setminus \left\{0\right\}$
\item If $dim A^1 = 3$, let $A^1=\langle e_1,e_2,e_3 \rangle$, then we have :
\begin{enumerate}
\item $\AKMSbracket{e_2,e_3,e_4}=e_1 ; \AKMSbracket{e_2,e_4,e_5}=-e_2 ; \AKMSbracket{e_3,e_4,e_5}=e_3$.
\item $\AKMSbracket{e_2,e_3,e_4}=e_1 ; \AKMSbracket{e_3,e_4,e_5}=e_3 + \alpha e_2 ; \AKMSbracket{e_2,e_4,e_5}=e_3 ; \AKMSbracket{e_1,e_4,e_5}=e_1$.
\item $\AKMSbracket{e_2,e_3,e_4}=e_1 ; \AKMSbracket{e_3,e_4,e_5}=e_3 ; \AKMSbracket{e_2,e_4,e_5}=e_2 ; \AKMSbracket{e_1,e_4,e_5}=2e_1$.
\item $\AKMSbracket{e_2,e_3,e_4}=e_1 ; \AKMSbracket{e_1,e_3,e_4}=e_2 ; \AKMSbracket{e_1,e_2,e_4}=e_3$.
\item $\AKMSbracket{e_1,e_4,e_5}=e_1 ; \AKMSbracket{e_2,e_4,e_5}=e_3 ; \AKMSbracket{e_3,e_4,e_5}=\beta e_2+(1+\beta)e_3$, $\beta \in \mathbb{K} \setminus \left\{0,1\right\}$.
\item $\AKMSbracket{e_1,e_4,e_5}=e_1 ; \AKMSbracket{e_2,e_4,e_5}=e_2 ; \AKMSbracket{e_3,e_4,e_5}=e_3$.
\item $\AKMSbracket{e_1,e_4,e_5}=e_2 ; \AKMSbracket{e_2,e_4,e_5}=e_3 ; \AKMSbracket{e_3,e_4,e_5}=s e_1 + t e_2 + u e_3$. And $3$-Lie algebras corresponding to this case with coefficients $s,t,u$ and $s',t',u'$ are isomorphic if and only if there exists a non-zero element $r \in K$ such that :
\[ s=r^3 s' ; t=r^2 t' ; u=r u' \]
\end{enumerate}
\item If $dim A^1 = 4$, let $A^1 = \langle e_1,e_2,e_3,e_4 \rangle$, then we have :
\begin{enumerate}
\item $\AKMSbracket{e_2,e_3,e_4}=e_1 ; \AKMSbracket{e_3,e_4,e_5}=e_2 ; \AKMSbracket{e_2,e_4,e_5}=e_3 ; \AKMSbracket{e_2,e_3,e_5}=e_4$.
\item $\AKMSbracket{e_2,e_3,e_4}=e_1 ; \AKMSbracket{e_1,e_3,e_4}=e_2 ; \AKMSbracket{e_1,e_2,e_4}=e_3 ; \AKMSbracket{e_1,e_2,e_3} = e_4$.
\end{enumerate}
\end{enumerate}
\end{theorem}
The $3$-Lie algebras which are induced by Lie algebras are given by the following proposition:
\begin{proposition}
Let $\mathbb{K}$ be an algebraically closed field of characteristic $0$. According to Theorems \ref{AKMSd4} and \ref{AKMSd5}, the $3$-Lie algebras induced by Lie algebras of dimension $d \leq 5$ are:
\begin{itemize}
\item $d=3$ Theorem \ref{AKMSd4}: 2.
\item $d=4$ Theorem \ref{AKMSd4}: 3.: a,b,c,d,e.
\item $d=5$ Theorem \ref{AKMSd5}: 1. 2. 3. 4.
\end{itemize}
\end{proposition}
\begin{proof}
By applying Proposition \ref{AKMS-3to2}, the algebras given in Theorem \ref{AKMSd4} 2. and 3. a,b,c,d,e and Theorem \ref{AKMSd5} 1.,2.,3. and 4. are all induced by Lie algebras, remaining algebras' derived algebras are not abelian and then they cannot be induced by Lie algebras (Theorem \ref{AKMSsolv2}). \qed
\end{proof}

\subsection{From Lie Algebras to $3$-Lie Algebras}
We list, below, all $3$ and $4$-dimensional Lie algebras and all $3$-Lie algebras they may induce, $3$-dimensional algebras are classified in \cite{Patera_et_al:invariants} and $4$-dimensional ones, partially, in \cite{degraaf4solvable}. For every Lie algebra, we compute all the trace maps and the induced $3$-Lie algebra using these trace maps.
\begin{theorem}[$3$-dimensional Lie algebras \cite{Patera_et_al:invariants}]
Let $\mathfrak{g}$ be a Lie algebra and $\left\{e_1,e_2,e_3\right\}$ a basis of $\mathfrak{g}$, then $\mathfrak{g}$ is isomorphic to one of the following algebras: (Remaining brackets are either obtained by skew-symmetry or zero)
\begin{enumerate}
\item The abelian Lie algebra $\AKMSbracket{x,y}=0, \forall x,y \in \mathfrak{g}$.
\item $ L(3,-1) : \AKMSbracket{e_1,e_2}=e_2$.
\item $ L(3,1) : \AKMSbracket{e_1,e_2}=e_3$.
\item $ L(3,2,a) : \AKMSbracket{e_1,e_3}=e_1 ; \AKMSbracket{e_2,e_3}=a e_2 ; 0 < |a| \leq 1$.
\item $ L(3,3) : \AKMSbracket{e_1,e_3}=e_1 ; \AKMSbracket{e_2,e_3}=e_1+e_2$.
\item $ L(3,4,a) : \AKMSbracket{e_1,e_3}= a e_1-e_2 ; \AKMSbracket{e_2,e_3} = e_1+a e_2 ; a\geq 0$.
\item $ L(3,5) : \AKMSbracket{e_1,e_2}=e_1 ; \AKMSbracket{e_1,e_3}=-2e_2 ; \AKMSbracket{e_2,e_3}=e_3$. 
\item $ L(3,6) : \AKMSbracket{e_1,e_2}=e_3 ; \AKMSbracket{e_1,e_3}=-e_2 ; \AKMSbracket{e_2,e_3}=e_1$.
\end{enumerate}
\end{theorem}

\begin{remark}
The classification given above is for the ground field $\mathbb{K} = \mathbb{R}$, if $\mathbb{K} = \mathbb{C}$ then $L\AKMSpara{3,2,\frac{x-i}{x+i}}$ is isomorphic to $L(3,4,x)$ and $L(3,5)$ is isomorphic to $L(3,6)$.
\end{remark}

\begin{theorem}[Solvable $4$-dimensional Lie algebras \cite{degraaf4solvable}]
Let $\mathfrak{g}$ be a solvable Lie algebra, and $\left\{e_1,e_2,e_3,e_4\right\}$ a basis of $\mathfrak{g}$, then $\mathfrak{g}$ is isomorphic to one of the following algebras: (Remaining brackets are either obtained by skew-symmetry or zero)

\begin{enumerate}
\item The abelian Lie algebra $\AKMSbracket{x,y}=0, \forall x,y \in \mathfrak{g}$.
\item $M^2$ : $\AKMSbracket{e_1,e_4}=e_1 ; \AKMSbracket{e_2,e_4}=e_2 ; \AKMSbracket{e_3,e_4}=e_3$.
\item $M^3_a$ : $\AKMSbracket{e_1,e_4}=e_1 ; \AKMSbracket{e_2,e_4}=e_3 ; \AKMSbracket{e_3,e_4}=-a e_2 + (a+1)e_3$.
\item $M^4$ : $\AKMSbracket{e_2,e_4}=e_3 ; \AKMSbracket{e_3,e_4}=e_3$.
\item $M^5$ : $\AKMSbracket{e_2,e_4}=e_3$.
\item $M^6_{a,b}$ : $\AKMSbracket{e_1,e_4}=e_2 ; \AKMSbracket{e_2,e_4}=e_3 ; \AKMSbracket{e_3,e_4}= ae_1+be_2+e_3$.
\item $M^7_{a,b}$ : $\AKMSbracket{e_1,e_4}=e_2 ; \AKMSbracket{e_2,e_4}=e_3 ; \AKMSbracket{e_3,e_4}= ae_1+be_2$ ($a=b \neq 0$ or $a=0$ or $b=0$).
\item $M^8$ : $\AKMSbracket{e_1,e_2}=e_2 ; \AKMSbracket{e_3,e_4}=e_4$.
\item $M^9_a$ : $\AKMSbracket{e_1,e_4}=e_1+ae_2 ; \AKMSbracket{e_2,e_4}=e_1 ; \AKMSbracket{e_1,e_3}=e_1 ; \AKMSbracket{e_2,e_3}=e_2$ ($X^2-X-a$ has no root in $\mathbb{K}$).
\item $M^{11}$ : $\AKMSbracket{e_1,e_4}=e_1 ; \AKMSbracket{e_3,e_4}=e_3 ; \AKMSbracket{e_1,e_3}=e_2$.
\item $M^{12}$ : $\AKMSbracket{e_1,e_4}=e_1 ; \AKMSbracket{e_2,e_4}=e_2 ; \AKMSbracket{e_3,e_4}=e_3 ; \AKMSbracket{e_1,e_3}=e_2$.
\item $M^{13}_a$ : $\AKMSbracket{e_1,e_4}=e_1+ae_3 ; \AKMSbracket{e_2,e_4}=e_2 ; \AKMSbracket{e_3,e_4}=e_1 ; \AKMSbracket{e_1,e_3}=e_2 $.
\item $M^{14}_a$ : $\AKMSbracket{e_1,e_4}=ae_3 ; \AKMSbracket{e_3,e_4}=e_1 ; \AKMSbracket{e_1,e_3}=e_2$. ($M^{14}_a$ is isomorphic to $M^{14}_b$ if and only if $a=\alpha^2 b$ for some $\alpha \neq 0$).
\end{enumerate}
\end{theorem}

In the following, we will give all the traces $\tau$ on the Lie algebras listed above, we add to this list two non-solvable Lie algebras of dimension $4$, and the induced $3$-Lie algebras: (for a Lie algebra $\mathfrak{g}$ $(e_i)_{1\leq i \leq dim \mathfrak{g}}$ is a basis of $\mathfrak{g}$, and for $x \in \mathfrak{g}$, $(x_i)_{1 \leq i \leq dim \mathfrak{g}}$ are its coordinates in this basis).

\begin{tabular}{ccc}\hline
\textbf{Lie algebra} & \textbf{Trace} & \textbf{Induced $3$-Lie algebra}\\ \hline
Abelian Lie algebra & All linear forms & Abelian $3$-Lie algebra \\ \hline
$L(3,-1)$ & $\tau (x) = t_1 x_1 + t_3 x_3$ & $\AKMSbracket{e_1,e_2,e_3} = t_3 e_2$ \\ \hline
$L(3,1)$ & $\tau (x) = t_1 x_1 + t_2 x_2$ & Abelian $3$-Lie algebra \\ \hline
\begin{tabular}{c}$L(3,2,a)$, $L(3,3)$ \\ $L(3,4,a)$ \end{tabular} & $\tau (x)=t_3 x_3$ & Abelian $3$-Lie algebra \\ \hline
$L(3,5)$, $L(3,6)$ & $\tau (x) = 0$ & Abelian $3$-Lie algebra \\ \hline
\begin{tabular}{c}$M^2$, $M^3_a$ \\ $M^6_{a,b}$; $M^7_{a,b}$\end{tabular} ($a \neq 0$) & $\tau (x) = t_4 x_4$ &  Abelian $3$-Lie algebra \\ \hline
$M^3_0$ & $\tau (x) = t_2 x_2 + t_4 x_4$ & $\AKMSbracket{e_1,e_2,e_4} = -t_2 e_1$ \\ \hline
$M^4$ & $\tau (x) = t_1 x_1 + t_2 x_2 + t_4 x_4$ & \begin{tabular}{c} $\AKMSbracket{e_1,e_2,e_4} = t_1 e_3$ \\ $\AKMSbracket{e_1,e_3,e_4}=t_1 e_3$ \\ $\AKMSbracket{e_2,e_3,e_4}= t_2 e_3$\end{tabular} \\ \hline
$M^5$ & $\tau (x) = t_1 x_1 + t_2 x_2 + t_4 x_4$ &  $\AKMSbracket{e_1,e_2,e_4}=t_1 e_3$ \\ \hline
$M^6_{0b}$ & $\tau (x) = t_1 x_1 + t_4 x_4$ & \begin{tabular}{c}$\AKMSbracket{e_1,e_2,e_4}=t_1 e_3$ \\ $\AKMSbracket{e_1,e_3,e_4}= t_1 (be_2+e_3)$\end{tabular}  \\ \hline
$M^7_{0b}$ & $\tau (x) = t_1 x_1 + t_4 x_4$ & \begin{tabular}{c} $\AKMSbracket{e_1,e_2,e_4}=t_1 e_3$ \\ $\AKMSbracket{e_1,e_3,e_4}= t_1 be_2$\end{tabular} \\ \hline
$M^8$ & $\tau (x) = t_1 x_1 + t_3 x_3$ & \begin{tabular}{c} $\AKMSbracket{e_1,e_2,e_3}=t_3 e_2$ \\ $\AKMSbracket{e_1,e_3,e_4}= t_1 e_4$\end{tabular} \\ \hline
$M^9_a$ & $\tau (x) = t_3 x_3 + t_4 x_4$ & \begin{tabular}{c}$\AKMSbracket{e_1,e_3,e_4} = -t_3 (e_1 + a e_2) + t_4 e_1$ \\ $\AKMSbracket{e_2,e_3,e_4}=t_3 e_1 + t_4 e_2$\end{tabular} \\ \hline
$M^{11}$ & $\tau (x) = t_4 x_4$ & \begin{tabular}{c} $\AKMSbracket{e_1,e_3,e_4}=t_4 e_2$ \\ $\AKMSbracket{e_2,e_3,e_4}= t_4 e_1$\end{tabular} \\ \hline
\begin{tabular}{c}$M^{12}$, $M^{13}_a$ \\ $M^{14}_a$, $a\neq 0$\end{tabular} & $\tau (x) = t_4 x_4$ & $\AKMSbracket{e_1,e_3,e_4}=t_4 e_2$ \\ \hline
$M^{13}_0$ & $\tau (x) = t_3 x_3 + t_4 x_4$ & \begin{tabular}{c} $\AKMSbracket{e_1,e_3,e_4}=-t_3 e_1 + t_4 e_2$ \\  $\AKMSbracket{e_2,e_3,e_4}=-t_3 e_2$\end{tabular} \\ \hline
$M^{14}_0$ & $\tau (x) = t_3 x_3 + t_4 x_4$ & $\AKMSbracket{e_1,e_3,e_4}=t_4 e_2$ \\ \hline
$\mathfrak{gl}_2(\mathbb{K})$ & $\tau (x) = t_4 x_4$ & \begin{tabular}{c} $\AKMSbracket{e_1,e_2,e_4} = 2t_4 e_2$ \\ $\AKMSbracket{e_1,e_3,e_4}=-2t_4 e_3$ \\ $\AKMSbracket{e_2,e_3,e_4}=t_4 e_1$ \end{tabular} \\ \hline
$E_3 \times \mathbb{K}$ ($\mathbb{K}=\mathbb{R}$) & $\tau (x) = t_4 x_4$ & \begin{tabular}{c} $\AKMSbracket{e_1,e_2,e_4} = t_4 e_3$ \\ $\AKMSbracket{e_1,e_3,e_4}=-t_4 e_2$ \\ $\AKMSbracket{e_2,e_3,e_4}=t_4 e_1$ \end{tabular} \\ \hline
\end{tabular}
\\

Where $E_3$ denotes the $3$-dimensional Euclidean space equipped with the cross product.

%%%%%%%%%%%%%%%%%%%%%%%%%%%%%%%%%%%%%%%%%%%%%%%%%%%%%%

\section{Examples}\label{AKMS:Examples}
\subsection{Adjoint Representation $1$-Cocycles and Coboundaries}
Here, we give the set of $1$-cocycles/coboundaries of 4 chosen Lie algebras ($\mathfrak{gl}_2(\mathbb{K}$ and $M^4$, $M^5$ and $M^8$ in the classification above) in the classification above and $1$-cocycles/coboundaries of the induced algebras using a chosen trace map for each one, the computations were done using the computer algebra software Mathematica.

Shortly explained, the computation goes this way:

Let $(A,\AKMSbracket{.,.})$ be a Lie algebra of dimension $n$ with a basis $B=\{e_1,...,e_n\}$, $\tau$ a trace and $(A,\AKMSbracket{.,.,.}_\tau)$ the induced algebra.
Denote the structure constants of $(A,\AKMSbracket{.,.})$ with respect to this basis $B$ by $\AKMSpara{c_{ij}^k}_{1\leq i,j,k \leq n}$ and by $\AKMSpara{ct_{ijk}^q}_{1\leq i,j,k,q \leq n}$ those of $(A,\AKMSbracket{.,.,.}_\tau)$. The linear form $\tau$ is represented by the one-line matrix $T=\AKMSpara{t_i}_{1\leq i \leq n}$.
A given linear map $f : A \to A$ ($1$-cochain) may be represented by a $n\times n$ matrix, $Z=\AKMSpara{z_{ij}}_{1\leq i,j \leq n}$. In terms of structure constants, the condition for $f$ (represented by the matrix $Z$) to be a cocycle writes for the Lie algebra:
\[ \sum_{k=1}^n \AKMSpara{  c_{ij}^k z_{qk} - c_{kj}^q z_{ki} - c_{ik}^q z_{k,j}  } = 0 , \forall i,j,q, \] 
and for the induced ternary algebra
\[ \sum_{p=1}^n \AKMSpara{  ct_{ijk}^p z_{qp} - ct_{pjk}^q z_{pi} - ct_{ipk}^q z_{pj} - ct_{ijp}^q z_{pk}  }= 0, \forall i,j,k,q. \]

By solving these equations, we get a set of conditions, that we apply to $Z$ to get the matrices listed in the tables below, under "Cocycle" and "Ternary cocycle" respectively.

Matrices listed under "Coboundary" and "Ternary coboundary" are obtained by putting in column $j$ respectively $\AKMSbracket{y,e_j}$ or $\AKMSbracket{x,y,e_j}_\tau$, where $x=\AKMSpara{x_1,...,x_n}$ and $y=\AKMSpara{y_1,...,y_n}$, and $x_{ij}=x_i y_j - x_j y_i$.

\begin{itemize}
\item $\mathfrak{gl}_2(\mathbb{K})$:

\begin{tabular}{ccc} \hline
Cocycle & Coboundary & $dim H^1$ \\ \hline
$\begin{pmatrix}
0 & z_{12} & z_{13} & 0\\
-2z_{13} & z_{22} & 0 & 0\\
-2z_{12} & 0 & z_{22} & 0\\
0 & 0 & 0 & z_{44}
\end{pmatrix}$ &
 $\begin{pmatrix}
0 & -y_{3} & y_{2} & 0\\
-2y_{2} & 2y_{1} & 0 & 0\\
2y_{3} & 0 & -2y_{1} & 0\\
0 & 0 & 0 & 0
\end{pmatrix}$ &
 $1$ \\ \hline
Trace : & $\tau(x)=x_4$ & \\ \hline
Ternary cocycle & Ternary cobounary & $dim H^1_{\tau}$ \\ \hline
 $
\begin{pmatrix}
z_{11} & z_{12} & z_{13} & z_{14}\\
-2z_{13} & z_{22} & 0 & z_{24}\\
-2z_{12} & 0 & 2z_{11}-z_{22} & z_{34}\\
0 & 0 & 0 & -z_{11}
\end{pmatrix}
$ &
 $\begin{pmatrix}
0 & x_{34} & x_{42} & x_{23}\\
-2x_{42} & -2x_{14} & 0 & 2x_{12}\\
-2x_{34} & 0 & 2x_{14} & -2x_{13}\\
0 & 0 & 0 & 0
\end{pmatrix}$ &
 $1$ \\ \hline
\end{tabular}

\item $M^4$:

\begin{tabular}{ccc} \hline
Cocycle & Coboundary & $dim H^1$ \\ \hline
$
\begin{pmatrix}
z_{11} & z_{12} & 0 & z_{14}\\
z_{21} & z_{22} & 0 & z_{24}\\
-z_{21} & z_{32} & z_{22}+z_{32} & z_{34}\\
0 & 0 & 0 & 0
\end{pmatrix}
$ &
 $
\begin{pmatrix}
0 & 0 & 0 & 0\\
0 & 0 & 0 & 0\\
0 & -y_{4} & -y_{4} & y_{2}+y_{3}\\
0 & 0 & 0 & 0
\end{pmatrix}$ &
 $6$ \\ \hline
Trace : & $\tau(x)=x_{1}+x_{2}+x_{4}$ & \\ \hline
\end{tabular}

\begin{tabular}{ccc}
Ternary cocycle & & \\ \hline
$
\begin{pmatrix}
z_{11} & z_{12} & 0 & z_{14}\\
z_{21} & z_{11}-z_{12}+z_{21} & 0 & z_{24}\\
z_{31} & z_{32} & z_{11}-z_{12}-z_{31}+z_{32} & z_{34}\\
z_{41} & z_{41} & 0 & -z_{11}-z_{21}
\end{pmatrix}$
 & & \\ \hline
Ternary cobounary & $dim H^1_{\tau}$& \\ \hline
$
\begin{pmatrix}
0 & 0 & 0 & 0\\
0 & 0 & 0 & 0\\
x_{24}+x_{34} & x_{34}-x_{14} & -x_{14}-x_{24} & x_{12}+x_{13}+x_{23}\\
0 & 0 & 0 & 0
\end{pmatrix}$ &
 $6$ &\\ \hline

\end{tabular}

\item $M^5$:

\begin{tabular}{ccc} \hline
Cocycle & Coboundary & $dim H^1$ \\ \hline
$
\begin{pmatrix}
z_{11} & z_{12} & 0 & z_{14}\\
0 & z_{22} & 0 & z_{24}\\
z_{31} & z_{32} & z_{33} & z_{34}\\
0 & z_{42} & 0 & z_{33}-z_{22}
\end{pmatrix}
$ &
  $
\begin{pmatrix}
0 & 0 & 0 & 0\\
0 & 0 & 0 & 0\\
0 & -y_{4} & 0 & y_{2}\\
0 & 0 & 0 & 0
\end{pmatrix}
$ &
 $8$ \\ \hline
Trace : & $\tau(x)=x_{1}$ & \\ \hline
Ternary cocycle & Ternary cobounary & $dim H^1_{\tau}$ \\ \hline
$
\begin{pmatrix}
-z_{22}+z_{33}-z_{44} & z_{12} & 0 & z_{14}\\
z_{21} & z_{22} & 0 & z_{24}\\
z_{31} & z_{32} & z_{33} & z_{34}\\
z_{41} & z_{42} & 0 & z_{44}
\end{pmatrix}
$
 &
 $
\begin{pmatrix}
0 & 0 & 0 & 0\\
0 & 0 & 0 & 0\\
x_{24} & -x_{14} & 0 & x_{12}\\
0 & 0 & 0 & 0
\end{pmatrix}
$ &
 $9$ \\ \hline

\end{tabular}

\item $M^8$:

\begin{tabular}{ccc} \hline
Cocycle & Coboundary & $dim H^1$ \\ \hline
 $
\begin{pmatrix}
0 & 0 & 0 & 0\\
z_{21} & z_{22} & 0 & 0\\
0 & 0 & 0 & 0\\
0 & 0 & z_{43} & z_{44}
\end{pmatrix}
$ &
   $
\begin{pmatrix}
0 & 0 & 0 & 0\\
-y_{2} & y_{1} & 0 & 0\\
0 & 0 & 0 & 0\\
0 & 0 & -y_{4} & y_{3}
\end{pmatrix}
$ &
 $0$ \\ \hline
Trace : & $\tau(x)=x_{1}+x_{3}$ & \\ \hline
Ternary cocycle & Ternary cobounary & $dim H^1_{\tau}$ \\ \hline
$
\begin{pmatrix}
-z_{33} & 0 & z_{13} & 0\\
z_{21} & z_{22} & z_{23} & 0\\
z_{31} & 0 & z_{33} & 0\\
z_{41} & 0 & z_{43} & z_{44}
\end{pmatrix}
$
 &
$
\begin{pmatrix}
0 & 0 & 0 & 0\\
x_{23} & -x_{13} & x_{12} & 0\\
0 & 0 & 0 & 0\\
x_{34} & 0 & -x_{14} & x_{13}
\end{pmatrix}
$ &
 $4$ \\ \hline

\end{tabular}

\end{itemize}

\end{document}